\documentclass[a4paper,11pt]{article}
\usepackage{amssymb,amsmath,amsthm,amsfonts}
\usepackage{bookmark}
\usepackage{mathrsfs}
\usepackage{multirow}
\usepackage{graphicx}
\usepackage{authblk}
\usepackage{indentfirst}
\usepackage{multicol}
\usepackage{tabu}
\usepackage{url}
\usepackage{fancyhdr}
\usepackage[numbers]{natbib}
\usepackage[all]{xy}
\usepackage{mathtools}
\usepackage[english]{babel}
\usepackage{colortbl}
\usepackage{caption}
\usepackage{hyperref}\usepackage{subcaption}
\usepackage{tikz}\usepackage{float}\usepackage{wrapfig}

\usepackage{amsthm} 
\newtheoremstyle{example_nonitalic} 
  {} 
  {} 
  {\normalfont} 
  {} 
  {\bfseries} 
  {.} 
  { } 
  {} 

\theoremstyle{example_nonitalic} %
\newtheorem{example}{Example}[section]

\theoremstyle{plain} %
\newtheorem{theorem}{Theorem}[section]
\newtheorem{proposition}[theorem]{Proposition}

\newtheorem{definition}{Definition}[section]

\newcommand{\im}{{\mathrm{im}\hspace{0.1em}}}

\newcommand{\rank}{{\mathrm{rank}\hspace{0.1em}}}
\newcommand{\spann}{{\mathrm{span}_{\mathbb{K}}\hspace{0.1em}}}

\usepackage{xr}
\makeatletter
    \newcommand*{\addFileDependency}[1]{
    \typeout{(#1)}
    \@addtofilelist{#1}
    \IfFileExists{#1}{}{\typeout{No file #1.}}
    }
\makeatother

\title{Topological Sequence Analysis of Genomes: Delta Complex approaches}
\author[1,2]{Jian Liu}
\author[2]{Li Shen}
\author[2]{Dong Chen}
\author[2,3,4]{Guo-Wei Wei \thanks{Corresponding author: weig@msu.edu}}
\affil[1]{Mathematical Science Research Center, Chongqing University of Technology, Chongqing 400054, China}
\affil[2]{Department of Mathematics, Michigan State University, MI 48824, USA}
\affil[3]{Department of Electrical and Computer Engineering, Michigan State University, MI 48824, USA}
\affil[4]{Department of Biochemistry and Molecular Biology, Michigan State University, MI 48824, USA}

\makeatletter
    \renewcommand*{\@fnsymbol}[1]{\ensuremath{\ifcase#1\or \dagger\or *\or *\or
   \mathsection\or \else\@ctrerr\fi}}
\makeatother
\date{}

\begin{document}
    \maketitle
    \paragraph{Abstract}

 Algebraic topology has been widely applied to point cloud data to capture geometric shapes and topological structures. However, its application to genome sequence analysis remains rare. In this work, we propose topological sequence analysis (TSA) techniques by constructing $\Delta$-complexes and classifying spaces, leading to persistent homology, and persistent path homology on genome sequences. We also develop $\Delta$-complex-based persistent Laplacians to facilitate the topological spectral analysis of genome sequences. Finally, we demonstrate the utility of the proposed TSA approaches in phylogenetic analysis using Ebola virus sequences and whole bacterial genomes. The present TSA methods are more efficient than earlier TSA model, k-mer topology, and thus have a potential to be applied to other time-consuming sequential data analyses, such as those in linguistics, literature, music,  media, and social contexts.

\paragraph{Keywords}
     Topological sequence analysis, $\Delta$-complex, classifying space, genome sequences, phylogenetics.

\footnotetext[1]
{ {\bf 2020 Mathematics Subject Classification.}  	Primary  55N31; Secondary 62R40, 68Q07.
}

    \newpage
    \tableofcontents
    \newpage

\section{Introduction}\label{section:introduction}
Biological sequences, including the linear arrangements of nucleotides in DNA and RNA or amino acids in proteins, encode the genetic information necessary for the growth, development, and functioning of living organisms. Enormous efforts have been made in biological sequence analysis to understand the structure, function, and evolution of these sequences. To this end, a wide variety of computational and statistical methods have been developed for various tasks, including pairwise and multiple sequence alignment, motif and pattern discovery, genome assembly and annotation, phylogenetic analysis, variant calling and analysis, transcriptomics and RNA analysis, protein sequence analysis, epigenomics, and regulatory element identification.
Among these tasks, phylogenetic analysis studies the evolutionary relationships among biological species, individuals, or genes based on their genetic, morphological, or molecular data. It plays a critical role in understanding evolutionary history, the diversity of life, molecular epidemiology, taxonomy, systematics, and comparative genomics.

Both alignment-based and alignment-free methods have been developed for phylogenetic analysis. Alignment-based methods are crucial for studying a given species to identify mutations and genetic variants, whereas alignment-free methods are typically used for cross-species analysis, comparison, and ranking. Many alignment-free methods have been developed, including the natural vector method (NVM) \cite{deng2011novel,yu2013real}, Jensen-Shannon frequency profile (JS-FFP) \cite{jun2010whole,sims2009alignment}, Kullback-Leibler frequency profile (KL-FFP) \cite{vinga2003alignment}, Fourier Power Spectrum (FPS) \cite{hoang2015new}, and Markov KString \cite{qi2004whole}.  These methods analyze DNA sequences from different perspectives, focusing on frequency, statistical features, and patterns. NVM captures spatial distribution by analyzing k-mer positions, while JS-FFP and KL-FFP measure sequence similarity using divergence metrics based on base pair frequencies. FPS applies Fourier analysis to identify periodic components, and Markov KString models base dependencies to uncover local patterns. These approaches provide valuable insights into the complexity and structure of DNA sequences, with applications in genome, evolutionary biology, and phylogenetic analysis. While effective, these methods lack topological perspectives. By utilizing topological data analysis (TDA), one can extract novel insights into DNA sequences through persistent homology \cite{carlsson2009topology, edelsbrunner2008persistent} and persistent spectral theory \cite{chen2021evolutionary,wang2020persistent}.

Recent years have witnessed significant development in  topological data analysis. The strength of TDA lies in its ability to uncover the inherent topological structure of complex data, capturing global features, and providing strong resistance to noise. Due to these unique characteristics, TDA has found successful applications in fields such as molecular biology, materials science, and machine learning \cite{chen2024multiscale, sorensen2020revealing, pun2022persistent}. For example, topological deep learning (TDL) was first introduced in 2017 \cite{cang2017topologynet} and has become an emerging paradigm in data science since  then  \cite{papamarkou2024position}. More recently,   persistent spectral graph \cite{wang2020persistent}, quantum persistent homology \cite{ameneyro2024quantum,wee2023persistent,suwayyid2024persistent}, and persistent interaction topology \cite{liu2025persistent} have extended the scope of TDA.

TDA has been used to analyze viral sequences, focusing on sequence dissimilarity and representation \cite{chan2013topology,nguyen2022topological,dlotko2025combinatorial}. More recently, $k$-mer topology, a k-mer based topological sequence analysis (TSA) method, has been introduced to reveal the structure of genomic space, using persistent homology and/or persistent Laplacians to examine topological persistence and genetic distances across diverse genomic datasets \cite{hozumi2024revealing}.  The $k$-mer topology model was validated with large variety of phylogenetic tasks and outperformed other competing methods in genome analysis.

While a topological perspective may not be the most competitive for all biological task, it offers a novel approach and valuable insights for understanding biological structure, function, and dynamics. However, the application of algebraic topology to genome sequence analysis is still in its early stages. For example, $k$-mer topology is very time-consuming for analyzing large genomes. Therefore, there is a need to further develop topological  methods for sequential data.

In this work, we propose a $\Delta$-complex based topological sequence analysis  tailored for DNA or RNA sequences. Inspired by the concepts of $\Delta$-complexes and classifying spaces, we treated segments of a sequence as simplices, which serve as the basic building blocks of a $\Delta$-complex. Unlike simplices in conventional complexes, these simplices can consist of repeated elements, aligning more intuitively with the nature of sequences. Note that unlike k-mer topology, the proposed TSA method does not systematically analyze k-mers.

For a given sequence, we introduced a function defined on its segments, assigning values to simplices in the corresponding $\Delta$-complex. This approach facilitated the construction of a filtration of $\Delta$-complexes. Notably, the collection of simplices arising from the level sets of such functions does not always form a valid $\Delta$-complex, imposing certain constraints on the function. To address this, we considered the $\Delta$-closure of these simplices, ensuring that any function could yield a well-defined filtration of $\Delta$-complexes, thus enabling the definition of algebraic topology for sequences.

A particular case of interest is when the filtration function assigns the length of a sequence segment as its value. Under this construction, the resulting $\Delta$-complex corresponds exactly to the path complex, leading to the definition of persistent path homology \cite{chowdhury2018persistent} and persistent path Laplacian \cite{wang2023persistent}.

Additionally, for genome sequences composed of elements with an underlying group structure, such as DNA sequences consisting of the nucleotides A, C, G, and T (U), persistent homology can be constructed on the classifying space of the group, yielding intriguing new insights. Furthermore, we explored the integration of persistent Laplacian information into the multi-scale analysis of sequences. Since the kernel of the Laplacian operator is isomorphic to real-coefficient homology, its non-zero eigenvalues provide complementary information that has been shown to be critical in applications such as molecular biology. To demonstrate the computability of these TSA models, we provided examples such as "N-China-F," "N1-U.S.-P," "N2-U.S.-P," and "N3-U.S.-P" \cite{wang2020mutations}. These examples illustrate the versatility of our methods and their potential for broader applications in virus analysis and diagnostics.

In applications, we applied TSA  for phylogenetic analysis on two datasets: Ebola virus sequences and whole bacterial genomes. The Ebola dataset comprises 59 complete sequences distributed across five families, while the bacterial genome dataset includes 30 complete genomes spanning nine families. Using the spectral gap curves in dimension 1 as features, our approach achieved clustering that closely aligns with the expected classifications, albeit with some minor discrepancies, while maintaining relatively higher computational efficiency than k-mer topology.

The structure of our paper is organized as follows. The next section provides a review of the basic concepts of $\Delta$-complexes and classifying spaces. Section \ref{section:method} introduces TSA methods applied to genome sequences. In Section \ref{section:application}, we present a direct application. Finally, the paper concludes with a summary and discussion.

\section{Preliminaries}

\subsection{Simplicial complex and homology}

\subsection{$\Delta$-complex and its homology}

A $\Delta$-complex is a generalization of a simplicial complex used in algebraic topology to describe the combinatorial structure of spaces. Like simplicial complexes, $\Delta$-complexes are constructed from simplices, such as vertices, edges, triangles, and higher-dimensional analogs. However, they offer greater flexibility by allowing simplices to overlap in more general ways. This added flexibility highlights the broader applicability of $\Delta$-complexes in various areas of topology, where their relaxed conditions enable versatile representations of spaces.

A semi-simplicial set or $\Delta$-complex $K$ is a graded set $\{K_{n}\}_{n\geq 0}$ equipped with a collection of face maps $d_{i}:K_{n}\to K_{n-1}$ for $0\leq i\leq n$ such that
\begin{equation*}
  d_{i}d_{j}=d_{j}d_{i-1},\quad i<j.
\end{equation*}
In particular, a simplicial complex is a $\Delta$-complex.
Recall that given a manifold, we can obtain a simplicial complex through a simplicial triangulation. Similarly, by triangulating the manifold, we can obtain a $\Delta$-complex. Compared to the simplicial triangulation of a simplicial complex, representing the manifold's triangulation as a $\Delta$-complex is more concise and streamlined.

Let $K$ be a $\Delta$-complex, and let $\mathbb{K}$ be a field, such as the real number field $\mathbb{R}$, or the binary field $\mathbb{Z}/2$. Define $C_{n}(K)$ as the $\mathbb{K}$-linear space generated by the elements in $K_{n}$ for $n \geq 0$. We have a linear map $\partial_{n}:C_{n}(K)\to C_{n-1}(K)$ as follows:
\begin{equation*}
  \partial_{n} = \sum_{i=0}^{n}(-1)^{i}d_{i},\quad n \geq 0.
\end{equation*}
In particular, $\partial_{0}=0$. It can be verified that $\partial_{n-1}\circ\partial_{n}=0$. Thus, one has a chain complex
\begin{equation*}
\xymatrix{
  \cdots \ar@{->}[r]&C_{n}(K)\ar@{->}[r]^{\partial_{n}}& C_{n-1}(K)\ar@{->}[r]&\cdots \ar@{->}[r] & C_{2}(K)\ar@{->}[r]^{\partial_{2}}&C_{1}(K)\ar@{->}[r]^{\partial_{1}}& C_{0}(K).
  }
\end{equation*}
An $n$-cycle in $C_{n}(K)$ is an element $x$ such that $\partial_{n}x = 0$. The $\mathbb{K}$-linear space generated by all $n$-cycles is denoted by $Z_{n}(K)$. Similarly, an $n$-boundary is an element of the form $\partial_{n+1}y$ for some $y \in C_{n+1}(K)$. The $\mathbb{K}$-linear space generated by the $n$-boundaries is denoted by $B_{n}(K)$. By definition, every $n$-boundary is also an $n$-cycle, i.e., $B_{n}(K) \subseteq Z_{n}(K)$. The homology of the $\Delta$-complex $K$ is defined as
\begin{equation*}
  H_{n}(K) = Z_{n}(K)/B_{n}(K), \quad n \geq 0.
\end{equation*}
The rank of the homology group $H_{n}(K)$, known as the Betti number and denoted by $\beta_{n}$, represents the number of $n$-dimensional holes or independent cycles in the $\Delta$-complex $K$. These Betti numbers provide important topological features that capture the structure and complexity of the space at each dimension.

A 2-dimensional sphere can be represented as a $\Delta$-complex. Let the upper hemisphere be denoted by $\sigma$ and the lower hemisphere by $\tau$. Their intersection is a circle, denoted by $e_0$, and we choose a base point on this circle, labeled $v_0$. On the upper hemisphere, we select a point $v_1$ and connect $v_0$ to $v_1$ to form an edge $e_1$. Similarly, on the lower hemisphere, we select a point $v_2$ and connect $v_0$ to $v_2$, forming the edge $e_2$.
\begin{figure}[htb!]
  \centering
  \includegraphics[width=0.6\textwidth]{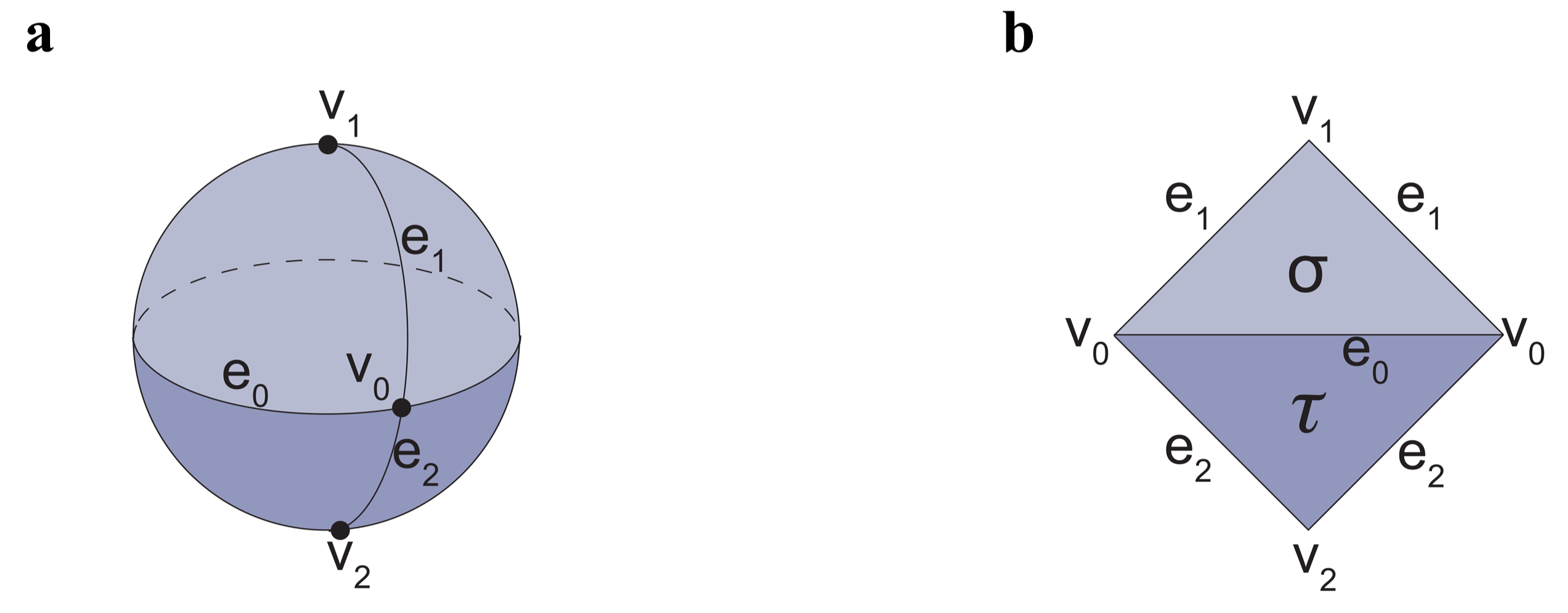}\\
  \caption{\textbf{a} Illustrating the triangulation of a 2-dimensional sphere.  \textbf{b} The $\Delta$-complex of the triangulation of a 2-dimensional sphere.}\label{figure:delta}
\end{figure}

In this way, we obtain a $\Delta$-complex with the following structure:
\begin{equation*}
  K_{0}=\{v_{0},v_{1},v_{2}\},\quad K_{1}=\{e_{0},e_{1},e_{2}\},\quad K_{2}=\{\sigma,\tau\}.
\end{equation*}
The face maps are given by:
\begin{equation*}
  \begin{split}
      & \partial_{0} \sigma =e_{0}, \quad\partial_{1} \sigma =\partial_{2} \sigma =e_{1}, \\
      & \partial_{0} \tau =e_{0}, \quad\partial_{1} \tau =\partial_{2} \tau =e_{2},\\
      & \partial_{0} e_{0}=\partial_{1} e_{0}=v_{0}, \\
      & \partial_{0} e_{1}=v_{0},\quad \partial_{1} e_{1}=v_{1}, \\
      & \partial_{0} e_{2}=v_{0}, \quad\partial_{1} e_{2}=v_{2} .
  \end{split}
\end{equation*}
Recall that if we represent a 2-dimensional sphere using a simplicial complex, it requires at least 4 vertices, 6 edges, and 4 faces. It will be more complex and also complicates the computation of its homology groups.

Now, consider the chain complex of of $K$. By definition, the chain complex of $K$ is given as
\begin{equation*}
     C_{2}(K) =\spann\{\sigma,\tau\}, \quad C_{1}(K) =\spann\{e_{0},e_{1},e_{2}\},\quad C_{0}(K) = \spann\{v_{0},v_{1},v_{2}\}.
\end{equation*}
The space of cycles of $K$ is
\begin{equation*}
  Z_{2}(K) = \spann\{\sigma-\tau\},\quad Z_{1}(K) = \spann\{e_{0}\}, \quad Z_{0}(K) = \spann\{v_{0},v_{1},v_{2}\}.
\end{equation*}
The space of boundaries of $K$ is given by
\begin{equation*}
  B_{1}(K) = \spann\{e_{0}\}, \quad B_{0}(K) = \spann\{v_{1}-v_{0},v_{2}-v_{0}\}.
\end{equation*}
Thus, the homology is
\begin{equation*}
  H_{n}(K)=\left\{
          \begin{array}{ll}
            \mathbb{K}, & \hbox{$n=0,2$;} \\
            0, & \hbox{otherwise.}
          \end{array}
        \right.
\end{equation*}
The result of this computation is consistent with the known homology of the 2-dimensional sphere.

\subsection{Classifying space}

A topological group is a group equipped with a topology such that both the group multiplication and the inverse operation are continuous maps. In particular, a discrete group can be viewed as a topological group with the discrete topology. Given a topological group $G$, we can define a $\Delta$-complex $EG$ as follows. The $n$-simplices of this $\Delta$-complex are given by $(n+1)$-tuples $(g_{0}, g_{1}, \dots, g_{n})$ where each $g_{i}$ is an element of $G$. The face map $d_{i}: (EG)_{n} \to (EG)_{n-1}$ is defined as
\begin{equation*}
  d_{i}(g_{0},g_{1},\dots,g_{n}) = (g_{0},\dots,g_{i-1},g_{i+1},\dots,g_{n})
\end{equation*}
for $0 \leq i \leq n$. It is worth noting that the elements in the $(n+1)$-tuples $(g_{0}, g_{1}, \dots, g_{n})$ are not required to be distinct, meaning that repetitions are allowed. For example, the word ``allowed'', letter ``l'' appears twice. This is a key reason why the construction forms a $\Delta$-complex rather than a simplicial complex, as simplicial complexes typically require that the vertices of each simplex be distinct.

The complex $EG$ is a contractible space, that is, it has the homotopy type of a point. It follows that the homology $H_{n}(EG)=\left\{
                                                                                                                                  \begin{array}{ll}
                                                                                                                                    \mathbb{K}, & \hbox{$n=0$;} \\
                                                                                                                                    0, & \hbox{otherwise.}
                                                                                                                                  \end{array}
                                                                                                                                \right.
$.
There is a natural group action of $G$ on the space $EG$, defined by
\begin{equation*}
  G\times EG \to EG,\quad g\cdot (g_{0},g_{1},\dots,g_{n}) =(gg_{0},gg_{1},\dots,gg_{n}).
\end{equation*}
The quotient space $BG = EG / G$, known as the orbit space of $G$ acting on $EG$, is called the \emph{classifying space} of $G$. And $EG$ is called the total space of the classifying space $BG$. It can be verified that $BG$ inherits a $\Delta$-complex structure from $EG$, making it a suitable model for computations in algebraic topology, particularly for studying the homotopy-theoretic properties of $G$. It is known that the classifying space has the homotopy type of the Eilenberg-MacLane space $K(G,1)$. Geometrically, the classifying space $BG$ classifies principal $G$-bundles over topological spaces. More precisely, for a topological group $G$, any principal $G$-bundle over a space $X$ corresponds to a map from $X$ to the classifying space $BG$. This map is unique up to homotopy, and the set of isomorphism classes of principal $G$-bundles over $X$ is in bijection with the homotopy classes of maps from $X$ to $BG$. Thus, $BG$ serves as a universal parameter space for such bundles.

Consider the Abelian group $G=\mathbb{Z}/2$. The classifying space of $G$ is $BG=\mathbb{R}P^{\infty}$, the infinite-dimensional real projective space. The homology of $\mathbb{R}P^{\infty}$ is
\begin{equation*}
  H_{n}(\mathbb{R}P^{\infty};\mathbb{Z}/2)=\mathbb{Z}/2,\quad n \geq 0.
\end{equation*}
This result aligns with the direct calculation of the classifying space $BG$. For example, when $n=1$ and considering group coefficients $\mathbb{K} = \mathbb{Z}/2$, we have
\begin{equation*}
\begin{split}
    & C_{2}(BG)=\spann\{(000),(001),(010),(100)\}, \\
    & C_{1}(BG)=\spann\{(00),(01)\},\\
    & C_{0}(BG)=\spann\{(0)\}.
\end{split}
\end{equation*}
Note that under the action of the group $\mathbb{Z}/2$, we have the identifications $(0) \sim (1)$, $(00) \sim (11)$, $(01) \sim (10)$, and so on. It follows that
\begin{equation*}
  Z_{1}(BG)=\spann\{(00),(01)\},\quad B_{1}(BG)=\spann\{(00)\}.
\end{equation*}
Thus, we obtain $H_{1}(\mathbb{R}P^{\infty}; \mathbb{K}) = \mathbb{K}$. The homology in other dimensions can be calculated similarly.

\section{Topological methods on sequences}\label{section:method}

In this section, we construct a filtration function on the space $EX$ in order to develop a theory of topological persistence for sequences.
The key idea is to encode the statistical or structural information arising from a sequence into a real-valued function, which in turn induces a filtration on $EX$.
This construction enables a systematic multi-scale analysis of the sequence via tools from topological data analysis, such as persistent homology and persistent Laplacians.

\subsection{Topological sequence analysis   of genome sequences}\label{section:persistent_homology}

Let $X$ be a non-empty finite set. Let $EX_{n}$ be the set of all the $(n+1)$-tuples $(x_{0}, x_{1}, \dots, x_{n})$ where each $x_{i}$ is an element in $X$. For any fixed $n\geq 1$ and  $0 \leq i \leq n$, we have the face map $d_{i}: EX_{n} \to EX_{n-1}$ defined by
\begin{equation*}
  d_{i}(x_{0},x_{1},\dots,x_{n}) = (x_{0},\dots,x_{i-1},x_{i+1},\dots,x_{n}).
\end{equation*}
One can verify that $d_{i}d_{j}=d_{j}d_{i-1}$ for $i<j$. The graded set $EX=\{EX_{n}\}_{n\geq 0}$ forms a $\Delta$-complex.

Let $f: EX \to \mathbb{R}$ be a real-valued function. We say that $f$ is \emph{face-preserving} if for any $x \in EX_{n}$ and any $0 \leq i \leq n$, the following inequality holds:
\begin{equation*}
f(d_{i}x) \leq f(x),
\end{equation*}
where $d_{i}$ denotes the $i$-th face map. Moreover, for real numbers $a \leq b$, there is an inclusion of $\Delta$-complexes
\begin{equation*}
  j^{a,b}:f^{-}(a)\hookrightarrow f^{-}(b).
\end{equation*}
This inclusion induces a morphism of homology groups
\begin{equation*}
  j^{a,b}_{n}:H_{n}(f^{-}(a))\rightarrow H_{n}(f^{-}(b)),\quad n \geq 0.
\end{equation*}
Thus, analogous to standard persistent homology theory, we define the image of the map $j^{a,b}_{n}$ as the persistent homology, representing homology generators that persist from time $a$ to time $b$.

\begin{definition}
The $(a,b)$-persistent homology of the function $f: EX \to \mathbb{R}$ is defined as
\begin{equation*}
H^{a,b}_{n}(X,f) \coloneqq \im \left( j^{a,b}_{n}: H_{n}(f^{-}(a)) \rightarrow H_{n}(f^{-}(b)) \right), \quad n \geq 0,
\end{equation*}
where $j^{a,b}_{n}$ is the map between homology groups induced by the inclusion $j^{a,b}: f^{-}(a) \hookrightarrow f^{-}(b)$.
\end{definition}
The persistent Betti number is denoted by $\beta_{n}^{a,b}=\rank H^{a,b}_{n}(X,f)$, which encodes the topological features of the sequence.

It is evident that in order to obtain the aforementioned persistent homology, the main challenge lies in finding a face-preserving function $f: EX \to \mathbb{R}$. In the following, we will construct such a function derived from a sequence. For consistency, from now on, the length of a tuple $(x_{0}, x_{1}, \dots, x_{k})$ or a sequence $x_{0}x_{1}\dots x_{k}$ is defined as $k$, which is the number of elements in the tuple or sequence minus 1.
Let $\xi$ be a sequence of finite length with elements in $X$. For any tuple $\sigma = (x_{0},x_{1},\dots,x_{k})\in EX$, \emph{a path $P(\sigma)$ of $\sigma$ in $\xi$} is a subsequence
\begin{equation*}
  x_{0}y_{1}\cdots y_{r_{1}} x_{1} y_{r_{1}+1} \cdots y_{r_{2}} x_{2}y_{r_{2}+1} \cdots  y_{r_{k-1}}x_{k-1} y_{r_{k-1}+1}\cdots y_{r_{k}+1} x_{k}.
\end{equation*}
We construct $f$ on a sequence $\xi$ as follows. For $\sigma = (x_0, \dots, x_k) \in EX$, define
\begin{equation*}
  \ell(\sigma) = \inf\limits\mathrm{length}(P(\sigma)),
\end{equation*}
where $P(\sigma)$ runs across all the paths of $\sigma$ in $\xi$. Here, $\mathrm{length}(P(\sigma))$ denotes the length of the path $P(\sigma)$. If there is no path of $\sigma$ in $\xi$, we set $\ell(\sigma)=\infty$.

For a tuple $\sigma = (x_{0},x_{1},\dots,x_{k})\in EX$, let $P(\sigma)$ be a shortest path of $\sigma$ in $\xi$. Then $P(\sigma)$ is also a path of the face $\partial_{i}(\sigma)$ of $\sigma$ in $\xi$. For $1\leq i\leq k-1$, we have that
\begin{equation*}
  \ell(\partial_{i}\sigma)\leq \mathrm{length}(P(\sigma))= \ell(\sigma).
\end{equation*}
For $\partial_{0}\sigma$ and $\partial_{k}\sigma$, we observe that $\ell(\partial_{i}\sigma) < \mathrm{length}(P(\sigma)) = \ell(\sigma)$. Thus, the map $\ell:EG\to \mathbb{Z}$ is a face-preserving function.

\begin{example}\label{example:cov}
SARS-CoV-2, the virus responsible for the COVID-19 pandemic, is a pathogen that affects humans through respiratory infection and has triggered a global health crisis. Accurate diagnosis of SARS-CoV-2 infection is crucial for controlling the spread of the virus, and in this process, primers play an essential role. Primers are short DNA fragments used in PCR tests to specifically amplify certain gene sequences of the virus, aiding in the rapid identification of its presence. Among the many primers developed, ``N-China-F'' stands out as the most widely used primer in China. This primer is designed to target the nucleocapsid (N) gene of SARS-CoV-2, a gene that plays a central role in the virus's replication and infectivity. Known for its high sensitivity and specificity, the ``N-China-F'' primer is extensively utilized in COVID-19 diagnostic testing, providing robust support for rapid detection and epidemic control. The expression of this sequence is as follows:
\begin{equation*}
\text{N-China-F: } \text{GGGG AAC TTCT CCTG CTAG AAT}.
\end{equation*}
This sequence is 22 nucleotides long composed of the nucleotide bases adenine (A), cytosine (C), guanine (G), and thymine (T), which are the building blocks of DNA. By a straightforward calculation, we have
\begin{equation*}
  \begin{split}
      & \mathrm{\ell(A)=\ell(C)=\ell(G)=\ell(T)= 0;}\\
      & \mathrm{\ell(AA)= \ell(AC)= \ell(AG)= \ell(AT)= 1;}\\
      & \mathrm{\ell(CA) = 2, \ell(CC) = 1, \ell(CG) = 2, \ell(CT) = 1;}\\
      & \mathrm{\ell(GA) = 1, \ell(GC) = 1, \ell(GG) = 1, \ell(GT) = 2;}\\
      & \mathrm{\ell(TA) = \ell(TC) = \ell(TG) = \ell(TT) = 1.}
  \end{split}
\end{equation*}
Similarly, we can compute pieces of greater length, for instance, $\ell(\mathrm{AGC})=10$, $\ell(\mathrm{CCC})=3$ and $\ell(\mathrm{CAC})=\infty$. As shown in Figure \ref{figure:bars}\textbf{a},
in dimension 0, there are four generator bars, all born at time 0, with three ending at time 1. In dimension 1, there are 10 generators born at time 1, with homology representatives given by \begin{equation*}
\begin{split}
    & \mathrm{AA,\quad CC,\quad GG,\quad TT,\quad AG+GA,\quad AT+TA,\quad CT+TC}, \\
    & \mathrm{AC+CT+TA},\quad \mathrm{GC+CT+TG},\quad \mathrm{TG+GA+AT}.
\end{split}
\end{equation*}
Here, the representatives $\mathrm{AA, CC, GG,}$ and $\mathrm{TT}$ are annihilated at time $\ell=2$ by the elements $\mathrm{AAG}$, $\mathrm{CCT}$, $\mathrm{GGA}$, and $\mathrm{CTT}$, respectively. The representatives $\mathrm{AG+GA}$ and $\mathrm{CT+TC}$ are annihilated by $\mathrm{AGA + AAG}$ and $\mathrm{TCT + CTT}$ at $\ell=2$, respectively.
Only the representative $\mathrm{AT+TA}$ is annihilated at time 3 by the element $\mathrm{GAT+GTA}$. Here, note that $\mathrm{\ell(GAT) = \ell(GTA) = 3}$.
Note that the generators $\mathrm{CG+GC}$ and $\mathrm{GT+TG}$ are born at $\ell=2$ but also die at $\ell=2$. This occurs because $\mathrm{CG+GC}$ and $\mathrm{GT+TG}$ are annihilated by the generators $\mathrm{TGC - CTG + CTC + CCT}$ and $\mathrm{TGC - GCT + CTC + CCT}$, respectively. Here, we have $\ell(\mathrm{TGC}) = \ell(\mathrm{GCT}) = \ell(\mathrm{CTG}) = 2$. Thus, $\mathrm{CG+GC}$ and $\mathrm{GT+TG}$ are not representatives of generators of persistent homology. Besides, the cycle $\mathrm{GC+CT+TG}$ is annihilated by $\mathrm{TGC+CTC+CCT}$ at time $\ell=2$.
On the other hand, the cycle $\mathrm{AC+CA}$ is born at time 2. But $\mathrm{AC+CA-AT-TA}$ is annihilated by $\mathrm{ACT-CTA}$, where $\ell(\mathrm{ACT}) = \ell(\mathrm{CTA})=2$. This means that $\mathrm{AC+CA}$ and $\mathrm{AT+TA}$ are in the same class at $\ell=2$. Similarly, the cycles $\mathrm{AC+CT+TA}$ and $\mathrm{AT+TA}$ are in the same homology class at $\ell=2$, differing by the boundary of $\mathrm{ACT}$. Likewise, the cycles $\mathrm{TG+GA+AT}$ and $\mathrm{AT+TA}$ are in the same class at $\ell=2$, differing by the boundary of $\mathrm{AGA+GAA-TAG}$. Thus, there is only one homology generator that persist up to $\ell=3$ with representative $\mathrm{AT+TA}$.

The above discussion shows the persistent Betti numbers
\begin{equation*}
  \beta_{0}^{0,1}=4,\quad \beta_{0}^{0,2}=\beta_{0}^{0,\infty}=1,\quad \beta_{1}^{1,2}=10,\quad \beta_{1}^{1,3}=1,\quad \beta_{1}^{1,4}= \beta_{1}^{1,\infty}= 0.
\end{equation*}
One can consider the further persistent topological features of the sequence.
\end{example}

\begin{figure}[htb!]
  \centering
  \includegraphics[width=0.9\textwidth]{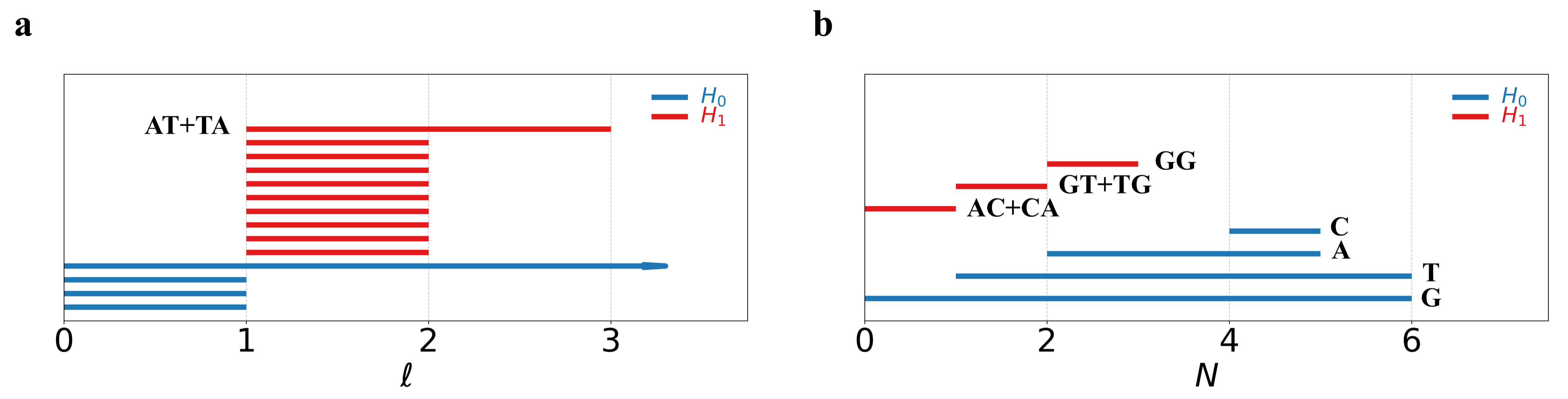}\\
  \caption{\textbf{a} The corresponding barcode based on the function $\ell:EG\to \mathbb{Z}$ on the sequence of ``N-China-F'' primer. \textbf{b} The corresponding barcode based on the function $N:EG\to \mathbb{Z}$ on the sequence of ``N-China-F'' primer.}\label{figure:bars}
\end{figure}

For a very long random sequence $\xi$ of length $L$ in a set $X$, consider a tuple $\sigma = (x_{0},x_{1},\dots,x_{k})\in EX$, where $k << L$. There almost always exists a contiguous subsequence $x_{0} x_{1} \cdots x_{k}$ in $\xi$. In this case, we can focus on the first occurrence of the path of $\sigma$, and thus define
\begin{equation*}
  \ell_{1}(\sigma) = \text{length}(P_{1}(\sigma)),
\end{equation*}
where $P_{1}(\sigma)$ denotes the first occurrence of the path of $\sigma$. Calculating $\ell_{1}(\sigma)$ is straightforward because we only need to detect the elements of $\sigma$ one by one in the sequence $\xi$, recording the position of the first and last elements to compute the length $\ell_{1}(\sigma)$. It can be verified that the map $\ell_{1}:EG \to \mathbb{N}$ is also a face-preserving function. This allows for the subsequent computation of persistent topological features.

\subsection{Persistent Laplacians on sequences}

The persistent Laplacians offer a spectral framework for analyzing topological features that persist across scales in a filtration of simplicial complexes \cite{chen2021evolutionary,liu2024algebraic,memoli2022persistent,wang2020persistent}. The construction of Laplacians on simplicial complexes is standard. In this section, we define the persistent Laplacians on the sequences and conduct the computations.

Let $K$ be a $\Delta$-complex. Let $C_{n}(K)$ be the linear space generated by the elements in $K_{n}$ over the real number field $\mathbb{R}$. Then $C_{\ast}(K)$ is a chain complex with the boundary map $\partial_{n}:C_{n}(K)\to C_{n-1}(K)$ for $n\geq 0$. There is a canonical inner product structure on $C_{\ast}(K)$ given by
\begin{equation*}
 \langle\sigma,\tau\rangle=\left\{
                                                                                                                                                           \begin{array}{ll}
                                                                                                                                                             1, & \hbox{$\sigma=\tau$;} \\
                                                                                                                                                             0, & \hbox{otherwise.}
                                                                                                                                                           \end{array}
                                                                                                                                                         \right.
\end{equation*}
The adjoint operator $(\partial_{n})^{\ast}:C_{n-1}(K)\to C_{n}(K)$ of $\partial_{n}$ is the unique linear operator satisfying
\begin{equation*}
  \langle \partial_{n}\sigma,\tau\rangle = \langle \sigma,(\partial_{n})^{\ast}\tau\rangle
\end{equation*}
for any $\sigma\in C_{n}(K)$ and $\tau\in C_{n-1}(K)$. The \emph{combinatorial Laplacian} of $K$ is defined by
\begin{equation*}
  \Delta_{n} = (\partial_{n})^{\ast}\circ\partial_{n}+\partial_{n+1}\circ(\partial_{n+1})^{\ast},\quad n\geq 1
\end{equation*}
and $\Delta_{0}=\partial_{1}\circ\partial_{1}^{\ast}$. It is clear that the operator $\Delta_{n}:C_{n}(K)\to C_{n}(K)$ is self-adjoint and semi-positive definite. Consequently, the eigenvalues of $\Delta_{n}$  are non-negative. The number of eigenvalues equal to $0$ corresponds to the Betti number of the $\Delta$-complex $K$. Furthermore, we have
\begin{equation*}
  \ker \Delta_{n} = H_{n}(K),\quad n\geq 0.
\end{equation*}
The nonzero eigenvalues of $\Delta_{n}$ encode the non-harmonic information of $K$.

Let $X$ be a non-empty finite set. Let $f: EX \to \mathbb{R}$ be a face-preserving real-valued function. Recall the sub $\Delta$-complex $f^{-}(a) = \{\sigma \in EX \mid f(\sigma) \leq a\}$ of $EX$. Let us denote $C_{n}^{a} = C_{n}(f^{-}(a))$. Then, for real numbers $a\leq b$, one has an inclusion of chain complexes
\begin{equation*}
   \mathfrak{j}_{\ast}^{a,b}:C_{\ast}^{a}\hookrightarrow C_{\ast}^{b}.
\end{equation*}

Let $C_{n+1}^{a,b}=\{x\in C_{n+1}^{b}|\partial_{n}^{b}x\in C_{n}^{a}\}$ be the set of elements in $C_{n+1}^{b}$ whose image of $\partial_{n}^{b}$ are in $C_{n}^{a}$. Here, $\partial_{n}^{b}$ is the boundary operator of $C_{\ast}^{b}$. Then we have a linear operator
\begin{equation*}
  \partial^{a,b}_{n}: C_{n+1}^{a,b}\to C_{n}^{a},\quad \partial^{a,b}_{n}(x)= \partial^{b}_{n}x.
\end{equation*}
The $(a,b)$-persistent Laplacian $\Delta^{a,b}_{n}:C_{a}\to C_{a}$ is given by
\begin{equation*}
  \Delta_{n}^{a,b} = \partial_{n+1}^{a,b}\circ (\partial_{n+1}^{a,b})^{\ast} + (\partial_{n}^{a})^{\ast}\circ \partial_{n}^{a}.
\end{equation*}

It is important to note that the kernel of the operator $\Delta_{n}^{a,b}$ is isomorphic to the $(a,b)$-persistent homology group $H_{n}^{a,b}(X,f)$ for any pair of parameters $a \leq b$. This implies that the harmonic part of the persistent Laplacian encapsulates the persistent homology information, while its non-harmonic part, i.e., the non-zero eigenvalues of the operator $\Delta_{n}^{a,b}$ reflects the non-harmonic spectral information of the simplicial complex as the filtration parameter evolves. The smallest positive eigenvalue $\lambda^{a,b}$ of the operator $\Delta_{n}^{a,b}$, known as the spectral gap, is a crucial spectral feature, reflecting certain intrinsic structural information of the sequence. In data analysis, computing the spectral feature at each time step, $\lambda^{a} = \lambda^{a,a}$, is particularly convenient and provides excellent visualization capabilities. In our data processing, we typically focus on the variation curve of the spectral feature, denoted as $\lambda(t) = \lambda^{t}$.

\begin{example}
In Example \ref{example:cov}, we analyzed the most widely used primer, ``N-China-F,'' in China. Figures \ref{figure:primers} presents the 0 to 3-dimensional Betti curves, as well as the spectral gap curves, for the ``N-China-F'' sequence. In addition, the most commonly used COVID-19 detection probes in the United States are the ``N1-U.S.-P,'' ``N2-U.S.-P,'' and ``N3-U.S.-P.'' These probes are widely utilized in PCR-based diagnostics for detecting SARS-CoV-2. Their corresponding sequences are given as follows\cite{wang2020mutations}:
\begin{equation*}
\begin{split}
    & \text{N1-U.S.-P: } \text{ACCCC GCATT ACGTT TG}, \\
    & \text{N2-U.S.-P: } \text{ACAAT TTGCC CCCAG CGCTT CAG}, \\
    & \text{N3-U.S.-P: } \text{ATCAC ATTGG CACCC GCAAT CCTG}.
\end{split}
\end{equation*}
These sequences target highly conserved regions of the SARS-CoV-2 genome, ensuring sensitivity and specificity in diagnostic applications. They are critical components of the CDC's recommended testing protocols in the U.S.
\begin{figure}[htb!]
  \centering
  \includegraphics[width=0.9\textwidth]{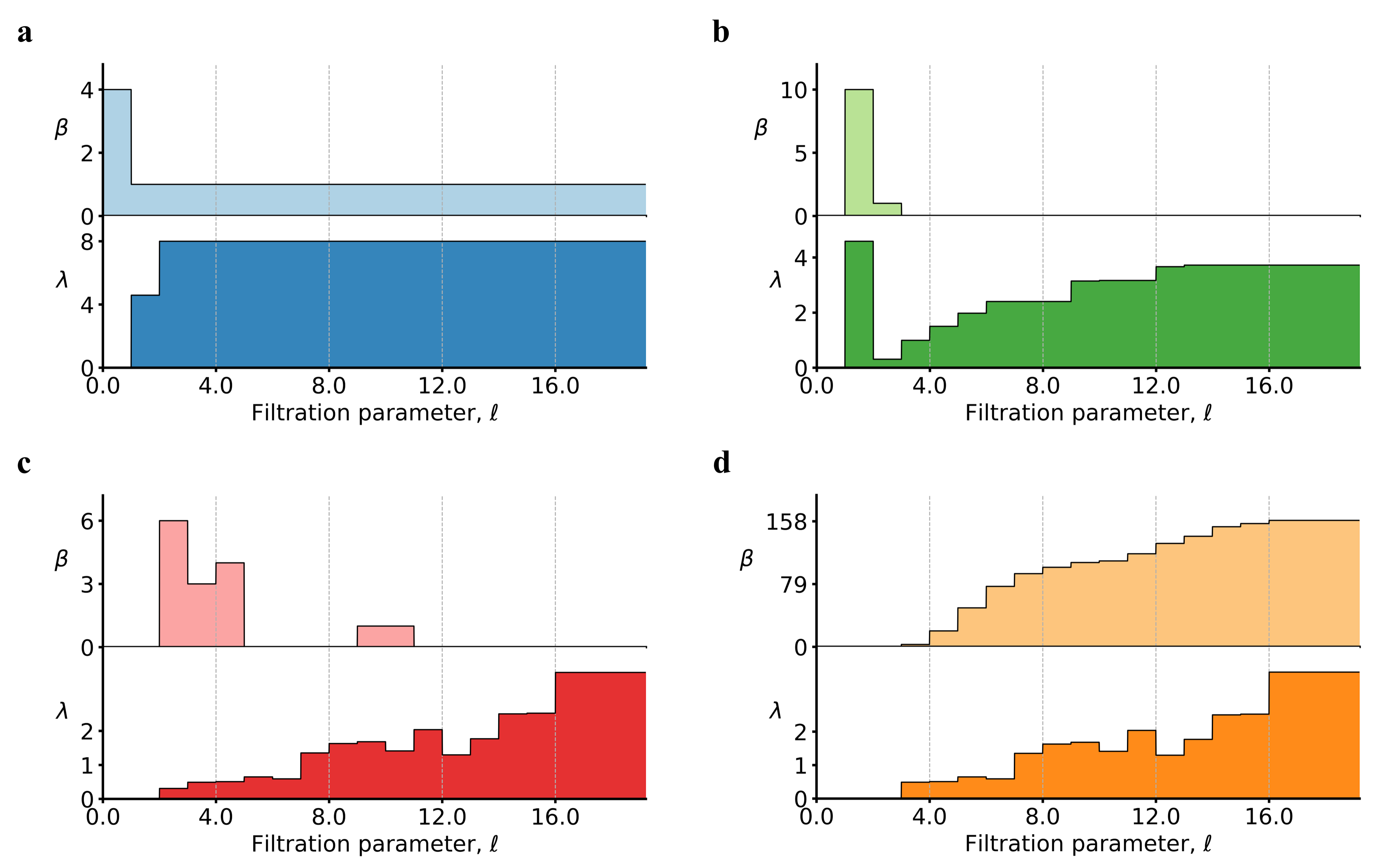}\\
  \caption{\textbf{a} The 0-dimensional Betti curve and spectral gap curve for the ``N-China-F'' sequence. \textbf{b} The 1-dimensional Betti curve and spectral gap curve for the ``N-China-F'' sequence. \textbf{c} The 2-dimensional Betti curve and spectral gap curve for the ``N-China-F'' sequence. \textbf{d} The 3-dimensional Betti curve and spectral gap curve for the ``N-China-F'' sequence.}\label{figure:primers}
\end{figure}
Figure \ref{figure:primers2} displays the Betti curves and spectral gap curves for the ``N1-U.S.-P,'' ``N2-U.S.-P,'' and ``N3-U.S.-P'' sequences. These curves reveal noticeable differences, either in 0-dimensional or 1-dimensional features, suggesting that each sequence captures distinct topological characteristics.
\begin{figure}[htb!]
  \centering
  \includegraphics[width=0.9\textwidth]{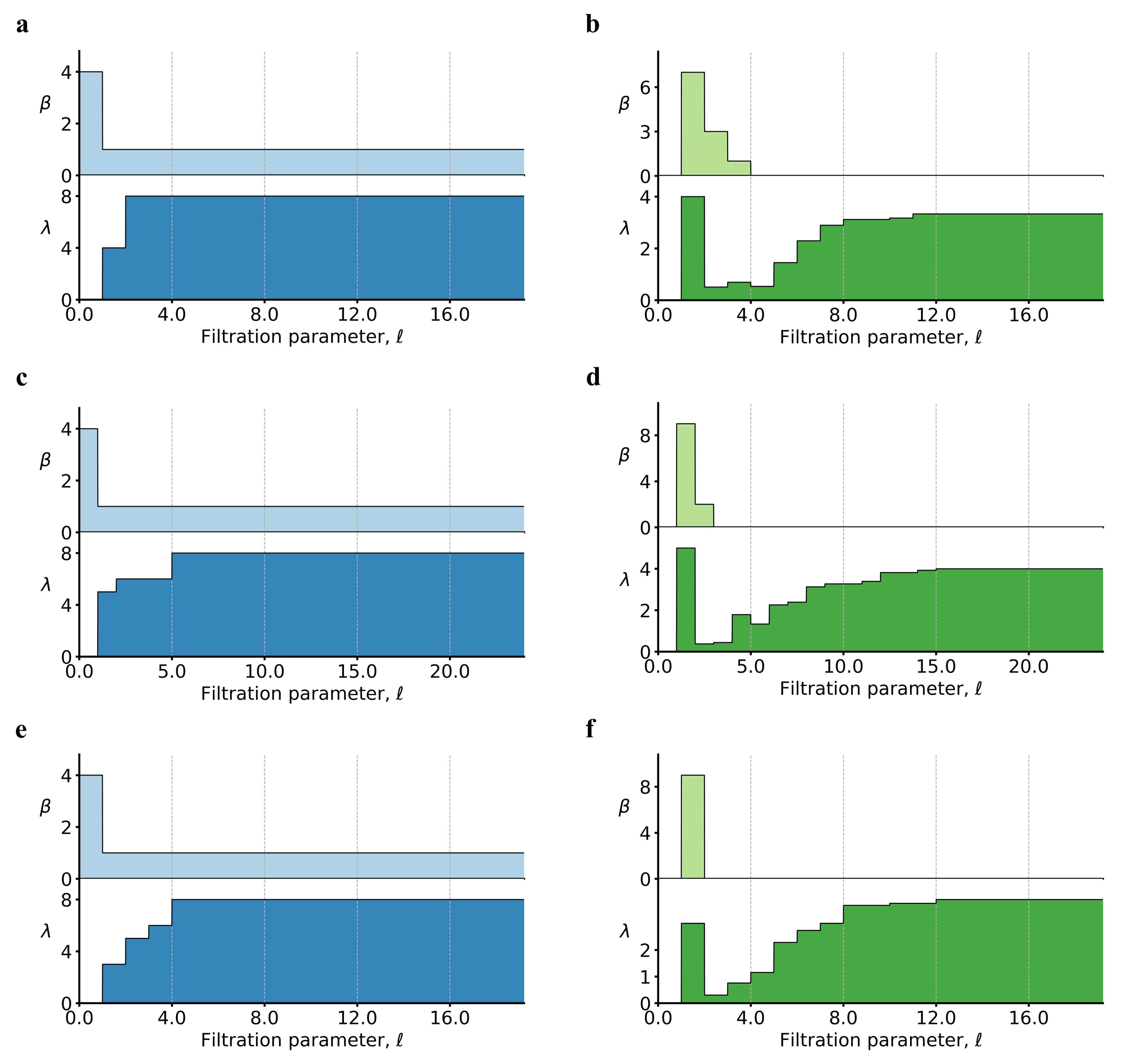}\\
  \caption{\textbf{a} The 0-dimensional Betti curve and spectral gap curve for the ``N1-U.S.-P'' sequence. \textbf{b} The 1-dimensional Betti curve and spectral gap curve for the ``N1-U.S.-P'' sequence. \textbf{c} The 0-dimensional Betti curve and spectral gap curve for the ``N2-U.S.-P'' sequence. \textbf{d} The 1-dimensional Betti curve and spectral gap curve for the ``N2-U.S.-P'' sequence. \textbf{e} The 0-dimensional Betti curve and spectral gap curve for the ``N3-U.S.-P'' sequence. \textbf{f} The 1-dimensional Betti curve and spectral gap curve for the ``N3-U.S.-P'' sequence.}\label{figure:primers2}
\end{figure}
\end{example}

\subsection{Persistent homology on $\Delta$-closure}

Let $f: EX \to \mathbb{R}$ be a function, which is not necessarily face-preserving. We can construct a modification of $f$ to obtain a face-preserving function. Define the function $\widetilde{f}: EX \to \mathbb{R}$ as follows. For any $\sigma\in EX$, set
\begin{equation*}
  \widetilde{f}(\sigma) = \max(\sup_{\tau} f(\tau), f(\sigma)),
\end{equation*}
where $\tau$ runs over all faces of $\sigma$. Then $\widetilde{f}$ is a face-preserving function.

Let $S$ be a collection of elements in $EX$. It is worth noting that $S$ does not have to be a $\Delta$-complex. Define
\begin{equation*}
  \Delta S = \{\tau \in EX \mid \text{$\tau$ is a face of $\sigma$ for some $\sigma \in S$}\},
\end{equation*}
which is called the $\Delta$-closure of $S$. Then $\Delta S$ is a $\Delta$-complex. For a real-valued function $f: EX \to \mathbb{R}$, let $f^{-}(a) \subseteq EX$ be a set. Then $\Delta f^{-}(a)$ is a $\Delta$-complex. It is easy to see that, for real numbers $a \leq b$, there is an inclusion of $\Delta$-complexes
\begin{equation*}
  \Delta f^{-}(a) \hookrightarrow \Delta f^{-}(b),
\end{equation*}
which induces a homomorphism on homology groups
\begin{equation*}
  H_{\ast}(\Delta f^{-}(a)) \rightarrow H_{\ast}(\Delta f^{-}(b)).
\end{equation*}

Let $H_{\ast}^{a,b}(f) = \im (H_{\ast}(\Delta f^{-}(a)) \rightarrow H_{\ast}(\Delta f^{-}(b)))$. Then $H_{\ast}^{a,b}(f)$ can be seen as a variant of persistent homology, which can be defined for non-face-preserving functions. In particular, if $f: EX \to \mathbb{R}$ is face-preserving, then the definition of $H_{\ast}^{a,b}(f)$ coincides with the $(a,b)$-persistent homology of the function $f: EX \to \mathbb{R}$ as introduced in Section \ref{section:persistent_homology}.

\begin{proposition}
For any real numbers $a \leq b$ and $n \geq 0$, we have
\begin{equation*}
  H_{n}^{a,b}(f) \cong H_{n}^{a,b}(\widetilde{f}),\quad n \geq 0.
\end{equation*}
\end{proposition}
\begin{proof}
If $\sigma\in \widetilde{f}^{-}(a)$, then we have $\widetilde{f}(\sigma)\geq a$. It follows that
\begin{equation*}
   \max(\sup_{\tau} f(\tau), f(\sigma))\geq a.
\end{equation*}
Here, $\tau$ runs over all faces of $\sigma$. If $f(\sigma)\geq a$, we have $\sigma\in f^{-}(a)\subseteq \Delta f^{-}(a)$. If $f(\tau)\geq a$ for some face $\tau$ of $\sigma$, we have $\tau\in \Delta f^{-}(a)$. It follows that $\widetilde{f}^{-}(a)\subseteq\Delta f^{-}(a)$. On the other hand, if $\sigma\in \Delta f^{-}(a)$, then there exists a face $\tau$ of $\sigma$ such that $f(\tau)\geq a$. Thus, one has $\sigma\in \widetilde{f}^{-}(a)$. This shows $\Delta f^{-}(a)\subseteq \widetilde{f}^{-}(a)$. Hence, we have $\widetilde{f}^{-}(a)=\Delta f^{-}(a)$.

Note that $\widetilde{f}^{-}(a)\hookrightarrow \widetilde{f}^{-}(b)$ and $\Delta f^{-}(a)\hookrightarrow \Delta f^{-}(b)$ induced by $a\to b$ are inclusions of $\Delta$-complexes. There is a natural isomorphism $F_{a}:H_{\ast}(\widetilde{f}^{-}(a))\to H_{\ast}(\Delta f^{-}(a))$. It implies the desired isomorphism.
\end{proof}

\subsection{Persistent path homology on sequences}

Let $\xi$ be a finite sequence with elements in $X$. For any tuple $\sigma \in EX$, we can regard $\sigma$ as a part of the sequence by converting $\sigma=(x_{0}, x_{1}, \dots, x_{n})$ into $\langle\sigma \rangle=x_{0}x_{1} \cdots x_{n}$. Define $N(\sigma)$ as the number of occurrences of $\sigma$ in a sequence $\xi$. This defines a function
\begin{equation*}
  N: EX \to \mathbb{Z}, \quad \sigma \mapsto N(\sigma).
\end{equation*}

While $N$ is a well-defined function, it is not necessarily face-preserving. Therefore, for any non-negative integer $a$, the set $N^{+}(a) = \{\sigma \in EX \mid N(\sigma) \geq a\}$ is not necessarily a $\Delta$-complex.

Recall the definition of a path complex: a path complex is a collection of sequences $P=\{P_{n}\}_{n\geq 0}$ on a finite set $X$ such that if $x_{0}x_{1} \cdots x_{n} \in P_{n}$, then $x_{0}x_{1} \cdots x_{n-1} \in P_{n-1}$ and $x_{1} \cdots x_{n} \in P_{n-1}$. Here, $P_{n}$ is the set of paths of $n+1$ elements. Let $P$ be a path complex on a finite set $X$. Let $\mathcal{A}_{n}(P)$ be the $\mathbb{K}$-linear space generated by the elements in $P_{n}$. Recall that $C_{\ast}(EX)$ is a chain complex.
The space $\mathcal{A}_{n}(P)$ can be regarded as a subspace of $C_{n}(EX)$ for any $n \geq 0$, by identifying the representations $(x_{0}, x_{1}, \dots, x_{n})$ and $x_{0}x_{1} \cdots x_{n}$. Note that $\partial \mathcal{A}_{n}(P)\subseteq C_{n}(EX)$. Let
\begin{equation*}
  \Omega_{n}(P) = \{u\in \mathcal{A}_{n}(P)|\partial u\in \mathcal{A}_{n-1}(P)\},\quad n\geq 0.
\end{equation*}
Then $\Omega_{\ast}(P)$ is a chain complex. The \emph{path homology} of $P$ is defined as
\begin{equation*}
  H_{n}(P)= H_{n}(\Omega_{\ast}(P)),\quad \geq 0.
\end{equation*}

Path homology can be used to study the topological structure of directed graphs, particularly in capturing key structures within complex directed networks \cite{grigor2012homologies,grigor2015cohomology}. It has found numerous applications in analyzing the structure of complex systems and networks \cite{chen2023path,chowdhury2018persistent}.

Let $a$ be a positive integer. Note that if $x_{0}x_{1} \cdots x_{n} \in N^{+}(a)$, then $N(x_{0}x_{1} \cdots x_{n-1}) \geq N(x_{1}x_{2} \cdots x_{n}) \geq a$. It follows that $x_{0}x_{1} \cdots x_{n-1} \in N^{+}(a)$ and $x_{1} \cdots x_{n} \in N^{+}(a)$. Thus, we have the following result.

\begin{proposition}
The subset $N^{+}(a)$ of $EX$ is a path complex.
\end{proposition}

For positive integers $a\leq b$, the $(a,b)$-persistent path homology of the sequence $\xi$ is
\begin{equation*}
  H^{a,b}_{n}(\xi) = \im \left(H_{n}(N^{+}(b))\to H_{n}(N^{+}(a))\right),\quad n\geq 0.
\end{equation*}
The persistent path homology can characterize, to some extent, the topological features and structural distribution of a sequence. Moreover, for sufficiently long sequences, we are more interested in the distribution of different pieces rather than the counting of those pieces. In this case, we define a real-valued function $p:EX \to \mathbb{R}$, which maps a piece $\sigma$ to the ratio $N(\sigma)/\ell(\xi)$, where $\ell(\xi)$ denotes the length of the sequence $\xi$. For a real number $a$, the construction $p^{+}(a) = \{\sigma \in EX \mid p(\sigma) \geq a\}$ forms a path complex. Furthermore, for real numbers $a \leq b$, the $(a,b)$-persistent path homology of the sequence $\xi$ can be obtained as
\begin{equation*}
H^{a,b}_{n}(\xi) = \im\left( H_{n}(p^{+}(b)) \to H_{n}(p^{+}(a)) \right), \quad n \geq 0.
\end{equation*}

\begin{example}\label{example:cov2}
Example \ref{example:cov} revisited. We first consider the function $N: EX \to \mathbb{Z}$ on the sequence
\begin{equation*}
  \xi=\mathrm{GGGGAACTTCTCCTGCTAGAAT}.
\end{equation*}
We obtain the following values:
\begin{equation*}
  \begin{split}
      & N(\mathrm{A})=5,N(\mathrm{C})=5,N(\mathrm{G})=6,N(\mathrm{T})=6, \\
      & N(\mathrm{AA})=2,N(\mathrm{AC})=1,N(\mathrm{AG})=1,N(\mathrm{AT})=1,\\
      & N(\mathrm{CA})=0,N(\mathrm{CC})=1,N(\mathrm{CG})=0,N(\mathrm{CT})=4,\\
      & N(\mathrm{GA})=2,N(\mathrm{GC})=1,N(\mathrm{GG})=3,N(\mathrm{GT})=0,\\
      & N(\mathrm{TA})=1,N(\mathrm{TC})=2,N(\mathrm{TG})=1,N(\mathrm{TT})=1.
  \end{split}
\end{equation*}
Note that in this example, the filtration parameter decreases from large to small values.
In dimension 0, there are two generators $\mathrm{G}$ and $\mathrm{T}$ born at $N=6$ and two generators $\mathrm{A}$ and $\mathrm{C}$ born at $N=5$. The generator $\mathrm{C}$ is annihilated by $\mathrm{CT}$ at $N=4$, and similarly, the generator $\mathrm{A}$ is annihilated by $\mathrm{GA}$ at $N=2$.

In dimension 1, the generator $\mathrm{GG}$ is born at $N=3$ and is annihilated by $\mathrm{GAA}$ at $N=2$. The generator $\mathrm{CT+TC}$ is born at $N=2$ and is annihilated by $\mathrm{CTC+TCC}$ at $N=1$. The generator $\mathrm{AC+CA}$ is born at $N=1$ and vanishes at $N=0$. Notably, all generators are annihilated at $N=0$. In this example, the generators $\mathrm{CT+TC}$, $\mathrm{CG+GC}$, and $\mathrm{GT+TG}$ belong to the same homology class. For instance, the boundary of $\mathrm{CTG - TGC}$ is $\mathrm{(CT + TC) - (CG + GC)}$, showing that $\mathrm{CT + TC}$ and $\mathrm{CG + GC}$ are in the same equivalence class. The corresponding barcode is shown in Figure \ref{figure:bars}\textbf{b}.
\end{example}

\subsection{Persistent homology on classifying spaces}
In certain cases, when considering sequences on a finite set $X$, the set $X$ may exhibit specific symmetries, particularly in the form of a finite group $G$. Such symmetries play a central role in shaping the combinatorial and topological properties of the sequences defined on $X$. To investigate these, we consider the corresponding classifying space, which serves as the foundation for constructing the persistent homology.

Let $G$ be a finite group. We have a classifying space $BG$ of $G$. Suppose $f:EG\to \mathbb{R}$ is a face-preserving function. We can obtain a real-valued function $\bar{f}:BG\to \mathbb{R}$ as follows. For $[\sigma]\in BG$, we set
\begin{equation}\label{equation:function}
  \bar{f}([\sigma]) = \frac{1}{|G|}\sum\limits_{g\in G}f(g\sigma).
\end{equation}
Here, $[\sigma]$ represents the equivalence class in the quotient space $BG=EG/G$.
For $\sigma'$ in the orbit $G\sigma$, we have $\sigma'=h\sigma$ for some $h\in G$. It follows that
\begin{equation*}
  \bar{f}([\sigma']) = \frac{1}{|G|}\sum\limits_{g\in G}f(gh\sigma) = \frac{1}{|G|}\sum\limits_{g'\in G}f(g'\sigma) =\bar{f}([\sigma]).
\end{equation*}
Thus the map $\bar{f}$ is well defined.
\begin{proposition}
Suppose $f:EG\to \mathbb{R}$ is a face-preserving function. Then the map $\bar{f}:BG\to \mathbb{R}$ is a face-preserving function.
\end{proposition}
\begin{proof}
Let $\tau$ be a face of $\sigma$. Then $g\tau$ is a face of $g\sigma$. It follows that $f(g\tau)\leq f(g\sigma)$.
Thus we have
\begin{equation*}
  \bar{f}([\tau]) =\frac{1}{|G|}\sum\limits_{g\in G}f(g\tau)\leq \frac{1}{|G|}\sum\limits_{g\in G}f(g\sigma) = \bar{f}([\sigma]).
\end{equation*}
This implies the desired result.
\end{proof}

The construction of the function $\bar{f}:BG\to \mathbb{R}$ implies that the filtration defined on the $\Delta$-complex $EG$ can induce a corresponding filtration on the $\Delta$-complex $BG$. Specifically, let $\bar{f}^{-}(a) = \{[\sigma] \in BG \mid f([\sigma]) \leq a\}$. For any real numbers $a \leq b$, there is an induced morphism
\begin{equation*}
  \bar{f}^{-}(a) \hookrightarrow \bar{f}^{-}(b)
\end{equation*}
between $\Delta$-complexes. Thus, we can define the $(a,b)$-persistent homology on the classifying space as
\begin{equation*}
  H^{a,b}_{n}(G,f) \coloneqq \im \left(H_{n}(\bar{f}^{-}(a)) \rightarrow H_{n}(\bar{f}^{-}(b)) \right), \quad n \geq 0.
\end{equation*}
The persistent homology of the classifying space incorporates certain symmetry information from the sequence, which arises from the action of the group $G$. The fundamental group of the classifying space $BG$ is isomorphic to $G$, and the persistence on $BG$ can be used to capture geometric information on $G$ induced by the filtration function. Therefore, the persistent homology of the classifying space can also be applied to study the topological persistence of the group, or the topological persistence of some symmetry structure.

\begin{example}
The sequence in Example \ref{example:cov} is revisited in this example for comparison. We consider the set $\{\mathrm{A,C,G,T}\}$ as the group $G = \mathbb{Z}/4$ by mapping $\mathrm{A} \to 0$, $\mathrm{C} \to 1$, $\mathrm{T} \to 2$, and $\mathrm{G} \to 3$. Then the elements of $BG$ of length $\leq 2$ are
\begin{equation*}
\begin{split}
    & \overline{0},\overline{00},\overline{01},\overline{02},\overline{03}, \\
    & \overline{000},\overline{001},\overline{002},\overline{003},\\
    & \overline{010},\overline{011},\overline{012},\overline{013},\\
    &  \overline{020},\overline{021},\overline{022},\overline{023},\\
    & \overline{030},\overline{031},\overline{032},\overline{033}.
\end{split}
\end{equation*}
By \eqref{equation:function}, we have the assignment of real numbers
\begin{equation*}
  \begin{split}
      & \mathrm{\bar{\ell}(\overline{0})= 0;}\\
      & \mathrm{\bar{\ell}(\overline{00})= 1, \bar{\ell}(\overline{01})= 1, \bar{\ell}(\overline{02})= 5/4,\bar{\ell}(\overline{03})= 3/2;}\\
      & \mathrm{\bar{\ell}(\overline{000}) = 11/4,\bar{\ell}(\overline{001}) = 5/2,\bar{\ell}(\overline{002}) = 7/2,\bar{\ell}(\overline{003}) = 11/2,;}\\
      & \mathrm{\bar{\ell}(\overline{010}) = 21/4,\bar{\ell}(\overline{011}) = 13/4,\bar{\ell}(\overline{012}) = 5/2,\bar{\ell}(\overline{013}) = 4;}\\
      & \mathrm{\bar{\ell}(\overline{020}) = 6,\bar{\ell}(\overline{021}) = 13/4,\bar{\ell}(\overline{022}) = 9/2,\bar{\ell}(\overline{023}) = \infty,;}\\
      & \mathrm{\bar{\ell}(\overline{030}) = \infty,\bar{\ell}(\overline{031}) = 23/4,\bar{\ell}(\overline{032}) = 2,\bar{\ell}(\overline{033}) = 17/4;}\\
  \end{split}
\end{equation*}
\begin{figure}[htb!]
  \centering
  \includegraphics[width=0.9\textwidth]{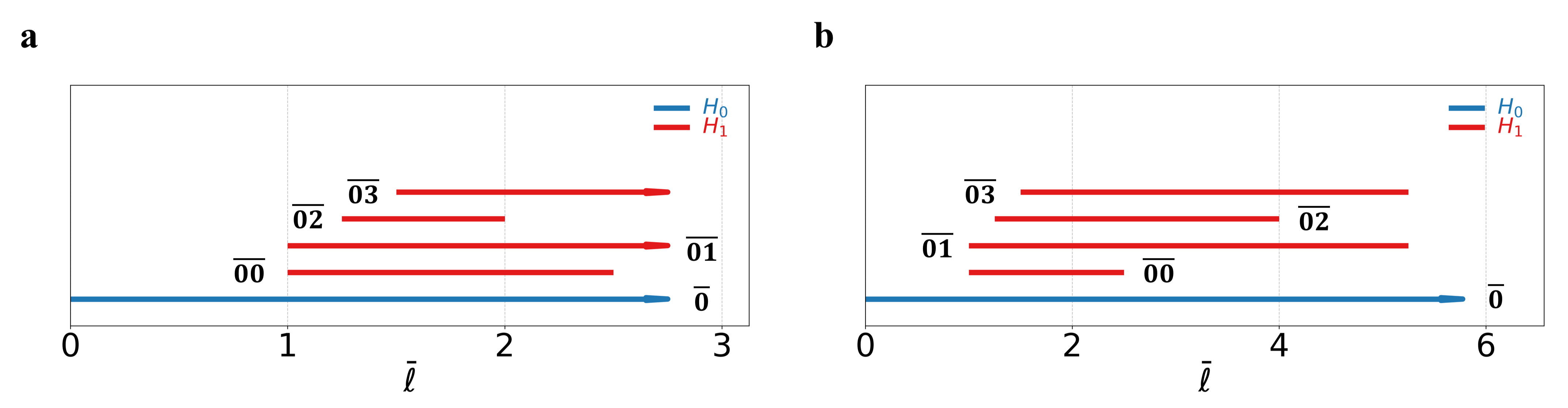}\\
  \caption{The corresponding barcode based on the function $\ell:EG\to \mathbb{R}$ on the sequence of ``N-China-F'' primer.}\label{figure:bars2}
\end{figure}
When we consider persistent homology with $\mathbb{Z}/2$ coefficients, the corresponding barcode is shown in Figure \ref{figure:bars2}\textbf{a}. When we consider persistent homology with real number coefficients, the corresponding barcode is shown in Figure \ref{figure:bars2}\textbf{b}. It is a common phenomenon that different coefficients lead to different persistent homology, as seen in the case of the homology generator $\overline{02}$ generated at $\bar{\ell} = 5/4$. At $\bar{\ell} = 2$, the boundary of the homology generator $\overline{032}$ is given by $2 \cdot \overline{03} - \overline{02}$ in real number coefficients. However, in $\mathbb{Z}/2$ coefficients, the boundary of $\overline{032}$ is exactly $\overline{02}$. This implies that at $\bar{\ell} = 2$, the generator $\overline{032}$ kills $\overline{02}$ in $\mathbb{Z}/2$ coefficients, but it does not kill $\overline{02}$ in real number coefficients.
\end{example}

\section{Applications}\label{section:application}

As discussed in the previous section, the frequency, length, and positional distribution of sequence fragments -- commonly referred to as motifs in biological contexts--are fundamental features for analyzing sequence data. These attributes not only capture essential structural information but also serve as a foundation for defining persistence parameters in computational studies. In this section, we illustrate how the TSA of motif frequency information within DNA sequences can be employed to perform biological clustering.

Frequency-based features are particularly attractive due to their computational efficiency. For instance, methods like Feature Frequency Profiles (FFP) capitalize on the $k$-mer frequency distributions in sequences to encode relevant information. While the computational cost of extracting frequency features is relatively low, the trade-off is a potential loss of accuracy in representing the nuanced properties of sequences. This study primarily aims to demonstrate the feasibility and effectiveness of employing frequency-based features in sequence analysis. Furthermore, this TSA  approach establishes a basis for the development of more sophisticated techniques to address the complexities inherent in sequence data.

The study of genome sequences is a fundamental tool for uncovering genetic information, exploring biological evolution, and understanding disease mechanisms. By analyzing DNA sequences, researchers can identify genetic variations, decode gene regulatory mechanisms, and advance the development of precision medicine.

In this work, to demonstrate the multi-scale TSA  of genome sequences, we apply phylogenetic analysis to datasets comprising the whole bacterial genomes and Ebola virus sequences. In this section, we demonstrate our method by applying the curve of the minimal positive eigenvalue (spectral gap) of the 1-dimensional Laplacian to perform a simple machine learning clustering analysis.

\begin{figure}[htb!]
  \centering
  \includegraphics[width=0.9\textwidth]{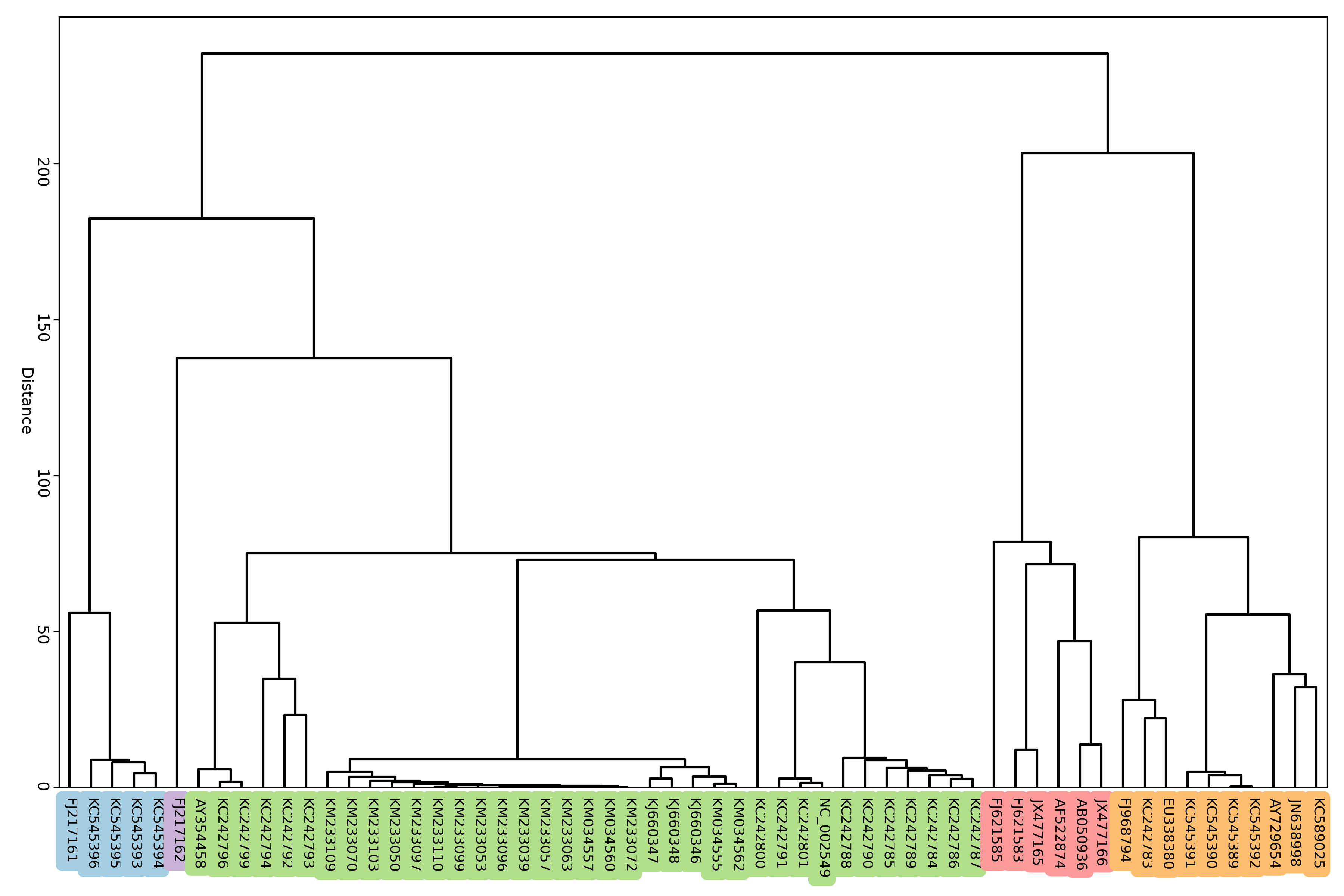}\\
  \caption{Phylogenetic tree for the Ebola virus generated using the TSA algorithm.}\label{figure:ebola}
\end{figure}

\paragraph{Topological distances}
In our algorithm, we directly use the Manhattan distance. Specifically, for two curves of spectral gaps $\lambda, \lambda': \mathbb{Z}_{\geq 0} \to \mathbb{R}$, their Manhattan distance is given by
\begin{equation*}
  d(\lambda, \lambda') = \sum\limits_{k \in \mathbb{Z}_{\geq 0}} |\lambda(k) - \lambda'(k)|.
\end{equation*}
In addition to the Manhattan distance, we can also consider other distances, such as the Euclidean distance
\begin{equation*}
d(\lambda, \lambda') = \sqrt{\sum_{k \in \mathbb{Z}_{\geq 0}} (\lambda(k) - \lambda'(k))^2}
\end{equation*}
and the Chebyshev distance
\begin{equation*}
d(\lambda, \lambda') = \max_{k \in \mathbb{Z}_{\geq 0}} |\lambda(k) - \lambda'(k)|.
\end{equation*}
More generally, the Minkowski distance can also be considered
\begin{equation*}
d(\lambda, \lambda') = \left( \sum_{k \in \mathbb{Z}_{\geq 0}} |\lambda(k) - \lambda'(k)|^p \right)^{1/p}.
\end{equation*}

\paragraph{TSA of Ebola virus genomes}

The Ebola virus dataset includes 59 complete sequences, each annotated with its corresponding classification into one of five categories: Bundibugyo virus (BDBV), Reston virus (RESTV), Ebola virus (EBOV), Sudan virus (SUDV), and Tai Forest virus (TAFV). As shown in Figure \ref{figure:ebola}, the phylogenetic tree for the Ebola virus, constructed using the TSA algorithm, demonstrates that this feature can classify the Ebola virus.

\begin{figure}[htb!]
  \centering
  \includegraphics[width=0.8\textwidth]{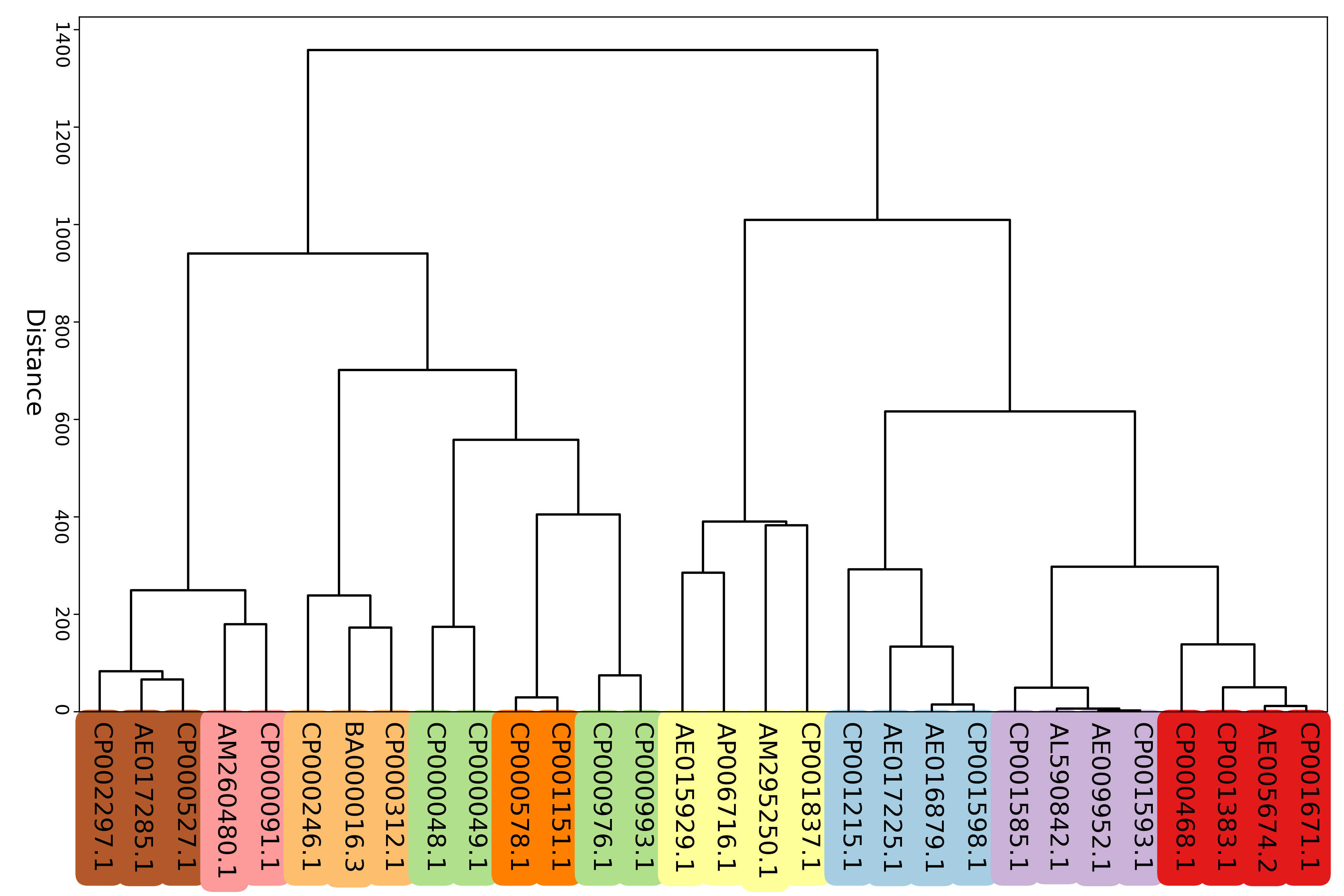}\\
  \caption{Phylogenetic tree for the whole bacterial genomes generated using the TSA algorithm.}\label{figure:bacteria}
\end{figure}

 \paragraph{TSA of bacterial genomes}
The dataset of whole bacterial genomes consists of 30 complete genomes from the families Bacillaceae, Borreliaceae, Burkholderiaceae, Clostridiaceae, Desulfovibrionaceae, Enterobacteriaceae, Rhodobacteraceae, Staphylococcaceae, and Yersiniaceae. The DNA sequence lengths of these genomes range from over 900,000 to more than 5 million base pairs. Such long sequences pose a significant challenge to computational methods. However, on our laptop configuration with an AMD Ryzen 5 5600H processor, the computation of the 0th, 1st, 2nd, and 3rd Betti curves and spectral gap curves for a 5 million base pair DNA sequence took about 6 seconds.
Table~\ref{table:time_compare} compares our TSA method with the $k$-mer topology approach \cite{hozumi2024revealing} for extracting topological features from 30 whole bacterial genomes. The results show that our TSA method is significantly faster, achieving comparable results in a fraction of the time, although its clustering performance is somewhat inferior.

The phylogenetic tree for the whole bacterial genomes is generally able to separate the families of bacterial genomes, as shown in Figure \ref{figure:bacteria}. However, for the genes CP000578.1 and CP001151.1, the classification results still exhibit some discrepancies. In fact, our example relies solely on frequency information, which inherently has certain limitations. Furthermore, we have only applied spectral gaps in dimension 1, which may prevent us from fully utilizing topological features in the training process. The primary contribution of our work is to propose a TSA method for constructing persistent homology on sequences. The effectiveness of training and data analysis will depend on the specific filtration functions and analytical techniques employed.

\begin{table}[ht]
  \centering
  \begin{tabular}{l|c}
    \hline
    Method & Total Time (s) \\
    \hline
    $k$-mer Topology ($k=3$) & 16895 \\
    $k$-mer Topology ($k=4$) & 5378 \\
    $k$-mer Topology ($k=5$) & 6596 \\
    TSA & \textbf{179} \\
    \hline
  \end{tabular}
  \caption{Comparison of total processing time (in seconds) for extracting topological features from 30 whole bacterial genomes using $k$-mer topology methods ($k = 3, 4, 5$)\cite{hozumi2024revealing}  and the present TSA method.}
  \label{table:time_compare}
\end{table}

\section{Conclusion}

Persistent homology and spectra-based persistent Laplacian  have been developed in the context of topological data analysis   over the past two decades, exerting a significant influence on fields such as computational science, biological sciences, physics,  chemistry, and deep learning. These methodologies have proven to be powerful tools for capturing multi-scale topological features in complex, high-dimensional, many-body, and nonlinear data. However, the application of algebraic topology to the analysis of sequential data, such as genome sequences, has not yet been sufficiently developed. The linear and combinatorial structures inherent in sequential data present unique challenges for which current topological frameworks are not fully equipped to address. The exploration of algebraic topology for sequence analysis holds considerable potential for advancing our understanding of sequence patterns, evolution, and dynamics. Bridging this gap requires the development of specialized filtration techniques and rigorous mathematical frameworks tailored to the structural and functional characteristics of sequence data. This represents a critical area for further research in topological data analysis  and topological machine learning.

In this work, we developed a topological sequence analysis (TSA) framework for analyzing sequences by treating sequence segments as simplices in the context of  $\Delta$-complexes. To establish persistent homology, we introduced a filtration of $\Delta$-complexes based on functions defined on sequence segments, ensuring validity through a $\Delta$-closure. For sequences with group structures, we extended this approach to classifying spaces, revealing new insights. Additionally, we integrated spectra-based persistent Laplacian information to enhance multi-scale analysis, demonstrating the framework's applicability through examples, such as the clustering of  DNA sequences. This method offers a versatile and effective approach for topological sequence analysis.

In the applications presented, we illustrated the utility of our TSA framework in  genome sequence clustering through the lens of topological persistence. By leveraging spectral gap curves of the 1-dimensional persistent Laplacian, we demonstrate the effectiveness of our approach in clustering biological datasets, such as Ebola virus sequences and whole bacterial genomes. The results show that while frequency-based features provide a computationally efficient means of capturing structural information, incorporating spectral features enhances the analytical depth. We show that the present TSA method is much faster than the $k$-mer topology model \cite{hozumi2024revealing}.
These studies validate the feasibility and scalability of our methods in handling diverse sequence data sets, while also highlighting potential limitations and opportunities for further refinement in leveraging topological features for biological sequence analysis.

Finally, the proposed TSA  approaches can be applied to   many other letter- or character-based sequential data, such as linguistics, literature, music, medical records, media records, and social chat records, with applications to surveillance, threat detection, and disease analysis.

\section*{Data and Code Availability}
The data and source code obtained in this work are publicly available in the Github repository: \url{https://github.com/WeilabMSU/Topological-persistence-on-sequences}.

\section{Acknowledgments}
This work was supported in part by NIH grants R01GM126189, R01AI164266, and R35GM148196, National Science Foundation grants DMS2052983, DMS-1761320, and IIS-1900473, Michigan State University Research Foundation, and  Bristol-Myers Squibb  65109. Jian was also supported by the Natural Science Foundation of China (NSFC Grant No. 12401080) and the start-up research fund from Chongqing University of Technology.

\bibliographystyle{plain}  
\bibliography{Reference}

\end{document}